\documentclass{article} %preprint, review  3p,twocolumn,review

\usepackage{hyperref}
\usepackage{amsmath} %% [tbtags] 使得多行公式的单一编号处于最后一行
\usepackage{amsfonts}

\usepackage{amssymb}
\usepackage[thmmarks,amsmath]{ntheorem}
\usepackage{tikz}       % draw
\usetikzlibrary{calc,arrows,decorations,shapes,decorations.pathmorphing,intersections,arrows.meta}

\usepackage{subcaption}
\usepackage{caption}
{   % define difinition environment
	\theoremstyle{nonumberplain}
	\theoremheaderfont{\bfseries}
	\theorembodyfont{\normalfont}
	\theoremsymbol{\ensuremath{\square}}
	\newtheorem{proof}{Proof}
}

\allowdisplaybreaks[3]

\newtheorem{theorem}{Theorem}[section]
\newtheorem{corollary}[theorem]{Corollary}

\newtheorem{proposition}[theorem]{Proposition}

\usepackage[bottom=30mm,left=25mm,right=25mm,top=30mm]{geometry}
\begin{document}

%\begin{frontmatter}
	
\title{Optimal cube factors of Fibonacci and matchable Lucas cubes\thanks{This work was supported by NSFC (Grant No. 11761064).}
}

\author{Xu Wang, Xuxu Zhao and Haiyuan Yao\footnote{Corresponding author.}%
	\\ {\footnotesize College of Mathematics and Statistics, Northwest Normal University, Lanzhou 730070, PR China}}
%\ead{xuwangac@hotmail.com}
%\author{Xuxu Zhao}
%\ead{zhaoxuxusd@hotmail.com}
%\author{Haiyuan Yao\corref{fn1}}
%\cortext[fn1]{Corresponding author}
%\ead{hyyao@nwnu.edu.cn}

%\address{College of Mathematics and Statistics, Northwest Normal University, Lanzhou, Gansu 730070, PR China}
\date{}

\maketitle

\begin{abstract}

The optimal cube factor of a graph, a special kind of component factor, is first introduced.
Furthermore, the optimal cube factors of Fibonacci and matchable Lucas cubes are studied; and some results on the Padovan sequence and binomial coefficients are obtained.

\textbf{Key words:} optimal cube factor; Fibonacci cube; matchable Lucas cube; Padovan sequence; Yang Hui triangle; Lucas triangle

\textbf{2010 AMS Subj. Class.:}  05C30; 05C70; 05A10; 11B83 %%MSCclass
\end{abstract}

%\begin{keyword}
%	$Z$-transformation digraph \sep finite distributive lattice \sep matchable Lucas cube \sep rank generating function \sep (maximal or disjoint) cube polynomial \sep degree (or indegree) spectrum polynomial
%	\MSC[2010] 11B39\sep 05C70, 06D05, 06A07
%\end{keyword}
%
%\end{frontmatter}

%\linenumbers

\section{Introduction}
%cube-factors, optimal cube-factors,  Component Factors \cite[Ch7]{bAkiyaK11} \cite[Ch4]{bYuL09} \cite{bBondyM08} \cite{bDaveyP02} \cite{bStanl11} \cite{aAkiyaK85}

%
%\section{Preliminaries}

The factorizations of graph were been studied by Akiyama and Kano \cite{aAkiyaK85,bAkiyaK11}, Yu and Liu \cite{bYuL09}, etc.
For a set $\{ A_1,A_2,A_3,\ldots\}$ of graphs, an $\{ A_1,A_2,A_3,\ldots\}$-factor of a graph $G$ is a spanning subgraph of $G$ such that each of its components is isomorphic to one of $\{ A_1,A_2,A_3,\ldots\}$ \cite[Chapter 4]{bYuL09}.
An $\{ A_1,A_2,A_3,\ldots\}$-factor is called a {\em cube factor} if each $A_i$ is isomorphic to a $k$-cube, where $i = 1,2,\ldots$ and $k$ is a non-negative integer.

Let $G$ be a graph, let $\mathbf{F}$ be a cube factor of $G$; let $q^F_k(G)$ denote the number of subgraphs that are isomorphic to $k$-cube in $\mathbf{F}$ and let $q^F_k(G) = 0$ if no $k$-cube in $\mathbf{F}$ exists. A cube factor $\mathbf{F}$ of $G$ %with spectum $\big(q^F_0,q^F_1,\ldots,q^F_m\big)$ 
is called as \emph{optimal},  if there is no cube factor $\mathbf{F}'$, such that $\sum_{k \ge 0} q^{F'}_k(G) < \sum_{k \ge 0} q^F_k(G)$.
Since the optimal cube factor is always denoted by $\mathbf{F}$, we will only write the optimal cube factor.
Moreover, let
\[
Q^F(G,x) = \sum_{k=0}^m q^F_k(G) x^k
\]
be optimal cube factor polynomial of $G$. 
Hence $Q^F(G,1)$ is the smallest for a graph $G$.

%By the definition of optimal cube factor polynomials, the following lemma is obvious.
%
%
%
%optimal $Q^F(G,1)$ minimum
%
%\begin{lemma}
%	The optimal cube factor polynomial of a graph is unique.
%\end{lemma}

The Padovan sequence $\{p_n\}$ is introduced by Stewart \cite[Chapter 8]{bStewa04}, where $p_0 = p_1 = p_2 = 1$, and $p_n = p_{n-2} + p_{n-3}$ for $n \ge 3$.

We first consider the optimal cube factor polynomial of Fibonacci cubes \cite{aHsu93}, obtain the generating function and coefficients of which, and the relation between Padovan sequence and the polynomials, moreover give formulae of Padovan numbers and the coefficients; in what follows, we explore matchable Lucas cubes \cite{aWangZY18c} in the same way.

\section{Optimal cube factors of Fibonacci cubes}

%\begin{definition}
%	cube factor, optimal cube factor
%\end{definition}

The Fibonacci sequence $\{F_n\}$ is defined as follows: $F_0=0$, $F_1=1$, and $F_n = F_{n-1} + F_{n-2}$ for $n \ge 2$. The Fibonacci cubes were introduced by Hsu \cite{aHsu93}, and lots of conclusions on Fibonacci cubes, include structure and enumerative properties, ware obtained \cite{%aGraviMSZ15,
	aKlavz13,%aKlavzMP11,
	aKlavM12,%aMolla12,
	aMunarZ02,%aPikeZ12,
	aZhangOY09,
	aZigerB13b}.
Let $\Gamma_n$ be the $n$-th Fibonacci cube, the dimension of maximum cube of $\Gamma_n$ is $\lfloor  n/2 \rfloor$ \cite{aKlavM12}, and the structure of $\Gamma_n$ $(n \ge 3)$ is shown as in Figure~\ref{fig:struc-fc}.

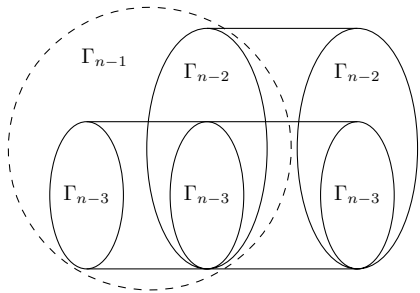
\begin{figure}[h]
	\centering
	\begin{tikzpicture}[scale=0.8,transform shape]
		\draw (-2,-2) -- (2.5,-2) (0,2) -- (2.5,2) (-2,0.448) -- (2.5,0.448);
		\draw (2.5,0) ellipse (1cm and 2cm)
		(0,0) ellipse (1cm and 2cm)
		(-2,-0.776) ellipse (0.612cm and 1.224cm)
		(0,-0.776) ellipse (0.612cm and 1.224cm)
		(2.5,-0.776) ellipse (0.612cm and 1.224cm);
		\draw[dashed] (-0.95,0) circle (2.35cm);
		
		\node at (-1.7,1.5) {$\Gamma_{n-1}$};
		\node at (0,1.25) {$\Gamma_{n-2}$};
		\node at (2.5,1.25) {$\Gamma_{n-2}$};
		\node at (-2,-0.776) {$\Gamma_{n-3}$};
		\node at (0,-0.776) {$\Gamma_{n-3}$};
		\node at (2.5,-0.776) {$\Gamma_{n-3}$};
	\end{tikzpicture}
	\caption{The structure of $\Gamma_n$ for $n \ge 3$}\label{fig:struc-fc}
\end{figure}

An easy induction gives the following theorem.
\begin{theorem}\label{th:opt-mm-F}
	An optimal cube factor of $\Gamma_n$ could be obtained as follows:
	
	For every $k \ge 0$, select the most $k$-dimensional cubes in $\Gamma_n$ under selecting the most $k+1$-dimensional cubes and deleting them from $\Gamma_n$.
\end{theorem}
\begin{proof}
	We proceed by induction on the $n$ of the Fibonacci cube.
	Verification for small values of $0 \le n < 3$ is trivial.
	Assume that the optimal cube factors of $j$ and $j+1$ can be obtained as described above, by the structure of $\Gamma_{j+3}$ as shown in Figure~\ref{fig:struc-fc}, we have
	\[
	\Gamma_{j+1} \mathbin{\square} P_2 \subseteq \Gamma_{j+3},
	\]
	and
	\[
	\Gamma_j \subseteq \Gamma_{j+1},
	\]
	where $G_1 \mathbin{\square} G_2$ denotes the Cartesian product of graphs $G_1$ and $G_2$, and $P_2$ is a path with two vertices.
	
	In addition, it is easy to see that the optimal cube factor of $\Gamma_{j+3}$ consists of the optimal cube factors of both $\Gamma_{j+1} \mathbin{\square} P_2$ and $\Gamma_j$.
	By the definition of Cartesian product, the optimal cube factor of $\Gamma_{j+1} \mathbin{\square} P_2$ corresponds to $\Gamma_{j+1}$.
	Therefore, the optimal cube factor of $\Gamma_{j+3}$ could be determined by $\Gamma_{j+1}$ and $\Gamma_j$.
	From the induction hypothesis, an optimal cube factor of $\Gamma_{j+3}$ can be obtained in the same way.
	
	This complete the induction.
\end{proof}

In fact, we just need to consider $0 \le k \le \lfloor \frac n2 \rfloor$ for $\Gamma_n$. Moreover, as a consequence of Theorem~\ref{th:opt-mm-F}, we have the recurrence relation of $q^F_k(\Gamma_n)$.

\begin{corollary}\label{th:rec-qF-F}
%	Let $q^F_k(\Gamma_n)$ denote the number of subgraphs that are isomorphic to $k$-cube in an optimal cube factor of $\Gamma_n$.
	For $n \ge 3$ and $0 \le k \le \lfloor \frac n2 \rfloor$,
	\[
	q^F_k(\Gamma_n) = q^F_{k-1}(\Gamma_{n-2}) + q^F_k(\Gamma_{n-3}).
	\]
\end{corollary}
%\begin{proof}
%	By induction. Verification for small values of $n$ is trivial.
%	
%	Next we assume that the conclusions for both $n-2$ and $n-3$ hold; that is, 
%	
%	By the structure of $\Gamma_n$ as shown in Figure~\ref{fig:struc-fc}, we have
%	\[
%	\Gamma_{n-2} \mathbin{\square} P_2 \subseteq \Gamma_n,
%	\]
%	where $P_2$ is a path with two vertices.
%	
%	In addition,
%	\[
%	\Gamma_{n-3} \subseteq \Gamma_{n-2}.
%	\]
%	
%	Therefore, after choosing the cubes of dimension greater than $k$ in an optimal cube factor, to get the number of $k$-cube, we must first consider $k$-cubes in $\Gamma_{n-2} \mathbin{\square} P_2$, i.e.\ the $k-1$-cubes in $\Gamma_{n-2}$, and then can choose it in the remaining $\Gamma_{n-3}$.
%	%	
%	%	\[
%	%	\Gamma_n = \Gamma_{n-2} \boxplus \Gamma_{n-2} \boxplus \Gamma_{n-3}
%	%	\]
%	%	implies feasible
%	%	
%	%	induction: $\Gamma_0$, $\Gamma_1$ and $\Gamma_2$
%	%	
%	%	Assume that $\Gamma_{n-3}$ and $\Gamma_{n-2}$, combining
%	%	
%	%	unique optimal cube factors
%\end{proof}

The first few of $Q^F(\Gamma_n,x)$ is listed as following.
\begin{align*}
	Q^F(\Gamma_0,x) &= 1 \\
	Q^F(\Gamma_1,x) &= x \\
	Q^F(\Gamma_2,x) &= 1+x \\
	Q^F(\Gamma_3,x) &= 1+x^2 \\
	Q^F(\Gamma_4,x) &= 2 x+x^2 \\
	Q^F(\Gamma_5,x) &= 1+2 x+x^3 %\\
%	Q^F(\Gamma_6,x) &= 1+3 x^2+x^3 \\
%	Q^F(\Gamma_7,x) &= 3 x+3 x^2+x^4 \\
%	Q^F(\Gamma_8,x) &= 1+3 x+4 x^3+x^4 \\
%	Q^F(\Gamma_9,x) &= 1+6 x^2+4 x^3+x^5
\end{align*}

The recurrence relation of $Q^F(\Gamma_n,x)$ is obtained easily.
\begin{proposition}\label{prop:rec-QF-F}
	For $n \ge 3$,
	\[
	Q^F(\Gamma_n,x) = xQ^F(\Gamma_{n-2},x) + Q^F(\Gamma_{n-3},x)
	\]
\end{proposition}
\begin{proof}
	By Theorem~\ref{th:rec-qF-F},
	\[
	Q^F(\Gamma_n,x) = \sum_{k=0}^m q^F_k(\Gamma_n) x^k = \sum_{k=0}^m q^F_{k-1}(\Gamma_{n-2}) x^k + \sum_{k=0}^m q^F_k(\Gamma_{n-3}) x^k = x\sum_{k=0}^m q^F_{k}(\Gamma_{n-2}) x^k + \sum_{k=0}^m q^F_k(\Gamma_{n-3}) x^k.
	\]
	The proof is completed.
\end{proof}

By recursion formula of Padovan sequences and Proposition~\ref{prop:rec-QF-F}, it is not difficult to verify the relation between $Q^F(\Gamma_n,1)$ and Padovan sequences.
\begin{corollary}\label{coro:rel-QF1-p-F}
	The sum of all coefficients of $Q^F(\Gamma_n,x)$ is the $(n+1)$-th Padovan number, that is
	\[
	Q^F(\Gamma_n,1) = p_{n+1}.
	\]
	And 
	\[
	Q^F(\Gamma_n,2) = F_{n+2}.
	\]
\end{corollary}

Furthermore, some tedious computation yields the generation function of $Q^F(\Gamma_n,x)$.
\begin{theorem}\label{th:gf-QF-F}
	The generation function of $Q^F(\Gamma_n,x)$ is given by
	\[
	\sum_{n=0}^\infty Q^F(\Gamma_n,x) y^n = \frac{1+y(x+y)}{1-y^2(x+y)}.
	\]
\end{theorem}
\begin{proof}
	By the recurrence relation of $Q^F(\Gamma_n,x)$, i.e.\ Proposition~\ref{prop:rec-QF-F},
	\begin{align*}
		\sum_{n=0}^\infty Q^F(\Gamma_n,x) y^n &= \sum_{n=3}^\infty Q^F(\Gamma_n,x) y^n + \sum_{n=0}^2 Q^F(\Gamma_n,x) y^n \\
		&= \sum_{n=3}^\infty xQ^F(\Gamma_{n-2},x) y^n + \sum_{n=3}^\infty Q^F(\Gamma_{n-3},x) y^n + \sum_{n=0}^2 Q^F(\Gamma_n,x) y^n \\
		&= xy^2\sum_{n=1}^\infty Q^F(\Gamma_n,x) y^n + xy^3\sum_{n=0}^\infty Q^F(\Gamma_n,x) y^n + \sum_{n=0}^2 Q^F(\Gamma_n,x) y^n \\
		&= y^2(x+y)\sum_{n=0}^\infty Q^F(\Gamma_n,x) y^n - xy^2 Q^F(\Gamma_0,x) + \sum_{n=0}^2 Q^F(\Gamma_n,x) y^n \\
		&= y^2(x+y)\sum_{n=0}^\infty Q^F(\Gamma_n,x) y^n + 1+y(x+y).
	\end{align*}
	Thus,
	\[
	(1-y^2(x+y))\sum_{n=0}^\infty Q^F(\Gamma_n,x) y^n = 1+y(x+y).
	\]
%	It follows the generating function of $Q^F(\Gamma_n,x)$.
\end{proof}

By the relation between $Q^F(\Gamma_n,1)$ and Padovan sequence, we have a consequence on Padovan sequence.
\begin{corollary}\label{coro:gf-pad}
	The generating function of Padovan sequence $\{p_n\}$ is
	\[
	\sum_{n=0}^\infty p_n y^n = \frac{1+y}{1-y^2(1+y)}.
	\]
\end{corollary}
\begin{proof}
	Combining Corollary~\ref{coro:rel-QF1-p-F} and Theorem~\ref{th:gf-QF-F},
	\[
	\sum_{n=0}^\infty p_n y^n = \sum_{n=1}^\infty p_n y^n + p_0 = y\sum_{n=0}^\infty p_{n+1} y^n + p_0 = y\sum_{n=0}^\infty Q^F(\Gamma_n,1) y^n + p_0 = \frac{y(1+y(1+y))}{1-y^2(1+y)}+1,
	\]
	which completes the proof.
\end{proof}

Let $[x^n]$ denote the \emph{coefficient extraction operator}, that is, $[x^n]g(x)$ denotes the coefficient of $x^n$ in the power series expansion of $g(x)$ \cite{bWilf94}.
A perfectly obvious property of this symbol, which we will use repeatedly, is $[x^n]\{x^mg(x)\} = [x^{n-m}]g(x)$.

Expanding the generating function of $Q^F(\Gamma_n,x)$ into power series, we have the formula of $q^F_k(\Gamma_n)$.
\begin{theorem}\label{th:qF-F}
	Let $\binom nk = 0$ if $n \notin \mathbb{Z}$. For $n \ge 0$,
%	\[
%	Q^F(\Gamma_n,x) = \sum_{j=\lfloor \frac{n+2}3 \rfloor}^{\lfloor \frac{n+1}2 \rfloor} \Bigg( \binom j{n-2j} x^{3j-n} + \binom{j}{n-2j+1} x^{3j-1-n} \Bigg).
%	\]
	\[
	q^F_{k}(\Gamma_n) = \binom{\frac{n+k}3}k + \binom{\frac{n+k+1}3}k.
	\]
	In detail, for $n \ge 0$,
	\[
	q^F_{k}(\Gamma_n) =
	\begin{cases}
	\binom mk, & \text{if } n+k = 3m, 3m-1; \\
	0, & \text{if } n+k = 3m+1.
	\end{cases}
	\]
\end{theorem}
\begin{proof}
	By Theorem~\ref{th:gf-QF-F}, we have
	\begin{align*}
		\frac1{1-y^2(x+y)} &= \sum_{j = 0}^\infty y^{2j}(x+y)^j \\
		&= \sum_{j = 0}^\infty \sum_{n-2j=0}^j \binom j{n-2j} x^{3j-n}y^n \\
		&= \sum_{n = 0}^\infty \sum_{j=\lfloor \frac{n+2}3 \rfloor}^{\lfloor \frac n2 \rfloor} \binom j{n-2j} x^{3j-n}y^n;
	\end{align*}
	namely
	\[
	[x^k][y^n]\frac1{1-y^2(x+y)} = \binom{\frac{n+k}3}k.
	\]
	Thus 
	\begin{align*}
		[x^k][y^n]\frac{y(x+y)}{1-y^2(x+y)} &= [x^k][y^{n-1}]\frac{x+y}{1-y^2(x+y)} \\
		&= [x^{k-1}][y^{n-1}]\frac1{1-y^2(x+y)} + [x^k][y^{n-2}]\frac1{1-y^2(x+y)} \\
		&= \binom{\frac{n+k-2}3}{k-1} + \binom{\frac{n+k-2}3}k \\
		&= \binom{\frac{n+k+1}3}k.
	\end{align*}
%	\begin{align*}
%		\frac{y(x+y)}{1-y^2(x+y)} &= y(x+y)\sum_{j = 0}^\infty y^{2j}(x+y)^j \\
%		&= \sum_{j = 0}^\infty y^{2j+1}(x+y)^{j+1} \\
%		&= \sum_{j = 0}^\infty \sum_{n-2j-1=0}^{j+1} \binom{j+1}{n-2j-1} x^{3j+2-n}y^n \\
%		&= \sum_{n = 0}^\infty \sum_{j=\lfloor \frac{n}3 \rfloor}^{\lfloor \frac{n-1}2 \rfloor} \binom{j+1}{n-2j-1} x^{3j+2-n}y^n \\
%		&= \sum_{n = 0}^\infty \sum_{j=\lfloor \frac{n}3 \rfloor+1}^{\lfloor \frac{n+1}2 \rfloor} \binom{j}{n-2j+1} x^{3j-1-n}y^n.
%	\end{align*}
%	Combining \eqref{eq:gf-QF.1-F} and \eqref{eq:gf-QF.2-F},
%	\[
%	\sum_{n=0}^\infty Q^F(\Gamma_n,x) y^n = \sum_{n = 0}^\infty \sum_{j=\lfloor \frac{n+2}3 \rfloor}^{\lfloor \frac{n+1}2 \rfloor} \Bigg( \binom j{n-2j} x^{3j-n} + \binom{j}{n-2j+1} x^{3j-1-n} \Bigg)y^n.
%	\]
%	Thus, we complete the proof.
\end{proof}

The coefficients of the first few $Q^F(\Gamma_n,x)$ listed are as shown in Table~\ref{tab:coef-QF-F}. It is not difficult to see that these coefficients are corresponding to the binomial coefficient. Furthermore, we have more general conclusions on $q^F_k(\Gamma_n)$ and expressions of the Padovan sequence.
\begin{table}[h]
	\centering
	\caption{The coefficients of first few $Q^F(\Gamma_n,x)$}\label{tab:coef-QF-F}
	\[
	\begin{matrix}%{*{7}{c}}
		1                                                                                \\
		0     & 1                                                                        \\
		1     & 1                                                                        \\
		1     & 0     & 1                                                                \\
		0     & 2     & 1                                                                \\
		1     & 2     & 0     & 1                                                        \\
		1     & 0     & 3     & 1                                                        \\
		0     & 3     & 3     & 0     & 1                                                \\
		1     & 3     & 0     & 4     & 1                                                \\
%		1     & 0     & 6     & 4     & 0     & 1                                        \\
%		0     & 4     & 6     & 0     & 5     & 1                                        \\
%		1     & 4     & 0     & 10    & 5     & 0     & 1                                \\
%		1     & 0     & 10    & 10    & 0     & 6     & 1                                \\
%		0     & 5     & 10    & 0     & 15    & 6     & 0     & 1                        \\
%		1     & 5     & 0     & 20    & 15    & 0     & 7     & 1                        \\
%		1     & 0     & 15    & 20    & 0     & 21    & 7     & 0     & 1                \\
%		0     & 6     & 15    & 0     & 35    & 21    & 0     & 8     & 1                \\
%		1     & 6     & 0     & 35    & 35    & 0     & 28    & 8     & 0     & 1        \\
%		1     & 0     & 21    & 35    & 0     & 56    & 28    & 0     & 9     & 1        \\
%		0     & 7     & 21    & 0     & 70    & 56    & 0     & 36    & 9     & 0     & 1 \\
%		1     & 7     & 0     & 56    & 70    & 0     & 84    & 36    & 0     & 10    & 1 \\
	\end{matrix}
	\]%%
\end{table}%

\begin{corollary}
	The number of terms with nonzero coefficients of $Q^F(\Gamma_n,x)$ is $\lfloor \frac{n+4}3 \rfloor$.
\end{corollary}

\begin{corollary}
	For $n \ge 0$ and $0 \le k \le \lfloor \frac n2 \rfloor$,
	\[
	q^F_k(\Gamma_{n-k}) =
	\begin{cases}
	\binom mk, & \text{if } n=3m-1 \text{ or } 3m; \\
	0, & \text{if } n=3m-2;
	\end{cases}
	\]
	and
	\[
	q^F_k(\Gamma_{n+2k}) =
	\begin{cases}
	\binom{k+m}k, & \text{if } n=3m-1,3m. \\
	0, & \text{otherwise.}
	\end{cases}
	\]
	In addition, the $n$-th Padovan number is calculated by
	\[
	p_n = \sum_{j=\lfloor \frac{n+1}3 \rfloor}^{\lfloor \frac n2 \rfloor} \binom{j+1}{n-2j}.
	\]
\end{corollary}
\begin{proof}
	It is immediate that $q^F_k(\Gamma_{n-k})$ and $q^F_k(\Gamma_{n+2k})$ by Theorem~\ref{th:qF-F}.
	
	And by Corollary~\ref{coro:rel-QF1-p-F},
	\[
	p_{n+1} = Q^F(\Gamma_n,1) = \sum_{j=\lfloor \frac{n+2}3 \rfloor}^{\lfloor \frac{n+1}2 \rfloor} \Bigg( \binom j{n-2j}  + \binom{j}{n-2j+1} \Bigg) = \sum_{j=\lfloor \frac{n+2}3 \rfloor}^{\lfloor \frac{n+1}2 \rfloor} \binom{j+1}{n-2j+1}.
	\]
%	\[
%	p_n = \sum_{j=\lfloor \frac{n+1}3 \rfloor}^{\lfloor \frac n2 \rfloor} \binom{j+1}{n-2j}
%	\]
	The proof is completed.
\end{proof}

\begin{corollary}
	For $n \ge 0$,
	\[
	\sum q^F_k(\Gamma_{n-k}) =
	\begin{cases}
	2^m, & \text{if } n=3m-1,3m; \\
	0, & \text{otherwise;}
	\end{cases} 
	\]
	and
	\[
	\sum q^F_k(\Gamma_{n-4k}) =
	\begin{cases}
	F_m, & \text{if } n=3m-1,3m; \\
	0, & \text{otherwise.}
	\end{cases}
	\]
\end{corollary}


\section{Optimal cube factors of matchable Lucas cubes}

The Lucas sequence $\{L_n\}$ is defined as follows: $L_0=2$, $L_1=1$, and $L_n = L_{n-1} + L_{n-2}$ for $n \ge 2$.
The Lucas triangle (see A029635 in OEIS \cite{Sloan19}) \cite{bKoshy01} $Y$ is shown in Table~\ref{tab:12-jy}, and the entry Lucas triangle is given by
\[
Y(n,k) = \binom nk+\binom{n-1}{k-1} 
%= \frac{n+k}{n}\binom nk = \frac{n+k}{k} \binom{n-1}{k-1}
= Y(n-1,k-1) + Y(n-1,k),
\]
where $0 \le k \le n$ and $\binom{-1}{-1} = 1$.

\begin{table}[!htbp]
	\centering
	\caption{The first six rows of Lucas triangle $Y$}\label{tab:12-jy}
	\[
	\begin{matrix}
	2 \\
	1 & 2 \\
	1 & 3 & 2 \\
	1 & 4 & 5 & 2 \\
	1 & 5 & 9 & 7 & 2 \\
	1 & 6 & 14 & 16 & 9 & 2 \\
	%	1 & 7 & 20 & 30 & 25 & 11 & 2 \\
	%	1 & 8 & 27 & 50 & 55 & 36 & 13 & 2 \\
	%	1 & 9 & 35 & 77 & 105 & 91 & 49 & 15 & 2 \\
	\end{matrix}
	\]
\end{table}

The matchable Lucas cubes are introduced by Wang et al.\ \cite{aWangZY18c} as follows: $|\Omega_0| = 1$, $\Omega_1 = P_2$, $\Omega_2 = P_3$, $\Omega_3 = P_4$, and the structure of $\Omega_n$ for $n \ge 5$ is shown as in Figure~\ref{fig:struc-mlc}, where $P_n$ is the path with $n$ vertices.
Furthermore, the first eight matchable Lucas cubes are shown as in Figure~\ref{fig:mlc8} \cite{aWangZY18c}.

%Let $q^F_k(\Omega_n) := t_k(\Omega_n)$ be maximum number of disjoint $k$-dimensional cubes in $\Omega_n$ after deleting every at least $k+1$-dimensional cubes (if exist), let $q^F_k(\Omega_n)=0$ if no $k$-dimensional cube exists after deleting every at least $k+1$-dimensional cubes (if exist), and let $Q^F(\Omega_n,x) = \sum_{k \ge 0} q^F_k(\Omega_n) x^k$ be the maximal disjoint cube polynomials. 

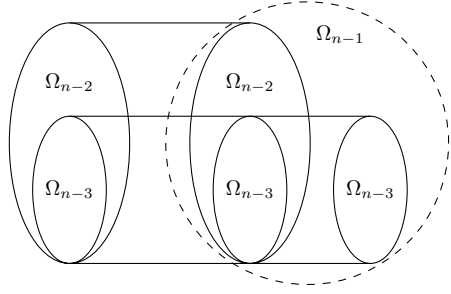
\begin{figure}[!h]
	\centering
	\begin{tikzpicture}[scale=0.8,transform shape]
		\draw (-3,2) -- (0,2) (-3,-2) -- (2,-2) (-3,0.448) -- (2,0.448);
		\draw (-3,0) ellipse (1cm and 2cm)
		(0,0) ellipse (1cm and 2cm)
		(2,-0.776) ellipse (0.612cm and 1.224cm)
		(0,-0.776) ellipse (0.612cm and 1.224cm)
		(-3,-0.776) ellipse (0.612cm and 1.224cm);
		\draw[dashed] (0.95,0) circle (2.35cm);
		
		\node at (1.5,1.8) {$\Omega_{n-1}$};
		\node at (0,1) {$\Omega_{n-2}$};
		\node at (-3,1) {$\Omega_{n-2}$};
		\node at (0,-0.776) {$\Omega_{n-3}$};
		\node at (-3,-0.776) {$\Omega_{n-3}$};
		\node at (2,-0.776) {$\Omega_{n-3}$};
	\end{tikzpicture}
	\caption{The structure of $\Omega_n$ for $n \ge 5$}\label{fig:struc-mlc}
\end{figure}

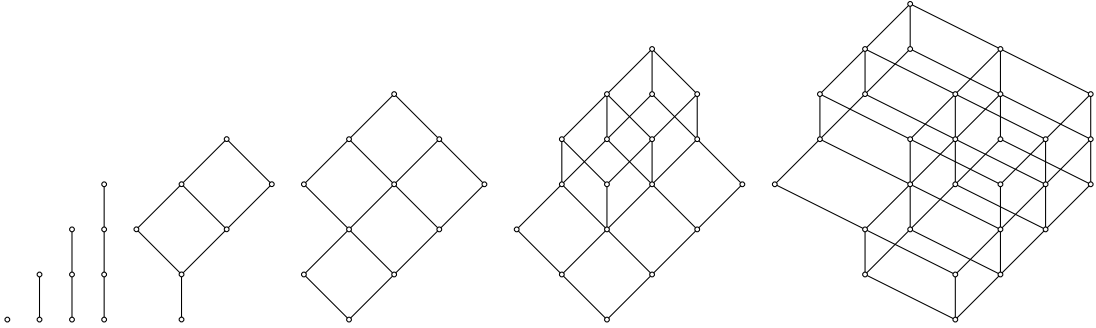
\begin{figure}[!h]
	\centering
	\begin{tikzpicture}[scale=0.6]
		\filldraw[fill=white] (0,0) circle (1.5pt);
	\end{tikzpicture}\quad
	\begin{tikzpicture}[scale=0.6]
		\draw (0,0) -- (0,1);
		\filldraw[fill=white] (0,0) circle (1.5pt);
		\filldraw[fill=white] (0,1) circle (1.5pt);
	\end{tikzpicture}\quad
	\begin{tikzpicture}[scale=0.6]
		\draw (0,0) -- (0,2);
		
		\foreach \j in {0,1,2}
		{
			\filldraw[fill=white] (0,\j) circle (1.5pt);
		}
	\end{tikzpicture}\quad
	\begin{tikzpicture}[scale=0.6]
		\draw (0,0) -- (0,3);
		
		\foreach \j in {0,...,3}
		{
			\filldraw[fill=white] (0,\j) circle (1.5pt);
		}
	\end{tikzpicture}\quad
	\begin{tikzpicture}[scale=0.6]
		\foreach \i in {0,1,2}
		{
			\draw (\i,\i) -- (\i-1,\i+1);
		}
		
		\draw (0,-1) -- (0,0) -- (2,2) (-1,1) -- (1,3);
		
		\foreach \i in {0,1,2}
		{
			\foreach \j in {0,1}
			{
				\filldraw[fill=white] (\i-\j,\i+\j) circle (1.5pt);
			}
		}
		\filldraw[fill=white] (0,-1) circle (1.5pt);
	\end{tikzpicture}\quad
	\begin{tikzpicture}[scale=0.6]
		\foreach \j in {0,1}
		{
			\foreach \i in {0,1,2}
			{
				\draw (\i-\j+1,\i+\j+1) -- (\i-\j,\i+\j+2)
				(\i-\j,\i+\j) -- (\i-\j+1,\i+\j+1);
			}
			\draw (-1+\j,3+\j) -- (-1+\j+1,3+\j+1);
		}
		
		\draw (0,0) -- (-1,1);
		
		\foreach \i in {0,1,2}
		{
			\foreach \j in {0,1,2}
			{
				\filldraw[fill=white] (\i-\j+1,\i+\j+1) circle (1.5pt);
			}
		}
		\filldraw[fill=white] (0,0) circle (1.5pt);
		\filldraw[fill=white] (-1,1) circle (1.5pt);
	\end{tikzpicture}\quad
	\begin{tikzpicture}[scale=0.6]
		\foreach \i in {0,1,2,3}
		{
			\foreach \j in {0,1}
			{
				\draw (\i-\j,\i+\j) -- (\i-\j-1,\i+\j+1);
			}
		}
		
		\foreach \i in {0,1,2}
		{
			\foreach \j in {0,1,2}
			{
				\draw (\i-\j,\i+\j) -- (\i-\j+1,\i+\j+1);
			}
		}
		
		\foreach \i in {0,1,2}
		{
			\draw (\i,\i+3) -- (\i-1,\i+1+3);
		}
		
		\foreach \i in {0,1}
		{
			\foreach \j in {0,1}
			{
				\draw (\i-\j,\i+\j+3) -- (\i-\j+1,\i+\j+1+3);
			}
		}
		
		\foreach \i in {0,1,2}
		{
			\foreach \j in {0,1}
			{
				\draw (\i-\j,\i+\j+2) -- (\i-\j,\i+\j+3);
			}
		}
		
		\foreach \i in {0,1,2}
		{
			\foreach \j in {0,1}
			{
				\filldraw[fill=white] (\i-\j,\i+\j+3) circle (1.5pt);
			}
		}
		
		\foreach \i in {0,1,2,3}
		{
			\foreach \j in {0,1,2}
			{
				\filldraw[fill=white] (\i-\j,\i+\j) circle (1.5pt);
			}
		}
	\end{tikzpicture}\quad
	\begin{tikzpicture}[scale=0.6]
		\foreach \i in {1,2,3}
		{
			\foreach \k in {0,1}
			{
				\foreach \j in {0,1}
				{
					\draw (\i-2*\k,\i+\k+\j) -- (\i-2*\k,\i+\k+\j+1)
					(-2*\k+1+\j,\i+\k+\j) -- (-2*\k+1+\j+1,\i+\k+\j+1);
				}
			}
		}
		
		\foreach \i in {1,2,3}
		{
			\foreach \j in {1,2,3}
			{
				\draw (\i,\i+\j-1) -- (\i-2,\i+\j);
			}
		}
		
		\foreach \i in {0,1}
		{
			\foreach \j in {0,1}
			{
				\draw %(\i,\i+\j) -- (\i-2,\i+\j+1)
				%        (\i-2*\j,\i+\j) -- (\i-2*\j,\i+\j+1)
				(-2*\j,\i+\j) -- (-2*\j+1,\i+\j+1);
			}
			\draw (0,\i) -- (-2,\i+1)
			(-2*\i,\i) -- (-2*\i,\i+1);
		}
		
		\foreach \i in {-1,0,1}
		{
			\foreach \j in {4,5}
			{
				\draw (\i,\i+\j) -- (\i-2,\i+\j+1);
			}
			\draw (\i-2,\i+4+1) -- (\i-2,\i+4+2);
			%   (-2+\i,\i+4) -- (\i-1,\i+5)
		}
		
		\foreach \i in {0,1}
		{
			\foreach \j in {4,5}
			{
				\draw (-3+\i,\i+\j) -- (\i-2,\i+\j+1);
			}
		}
		
		\draw (-2,2) -- (-4,3) -- (-3,4);
		
		\foreach \i in {1,2,3}
		{
			\foreach \j in {1,2,3}
			{
				\foreach \k in {0,1}
				{
					\filldraw[fill=white] (\i-2*\k,\i+\j+\k-1) circle (1.5pt);
				}
			}
		}
		
		\foreach \i in {-2,-1,0}
		{
			\foreach \j in {4,5}
			{
				\filldraw[fill=white] (\i-1,\i+\j+2) circle (1.5pt);
			}
		}
		
		\foreach \i in {0,1}
		{
			\foreach \j in {0,1}
			{
				\filldraw[fill=white] (-2*\j,\i+\j) circle (1.5pt);
			}
		}
		
		\filldraw[fill=white] (-4,3) circle (1.5pt);
	\end{tikzpicture}
	\caption{The first eight matchable Lucas cubes $\Omega_0$, $\Omega_1$, \dots, $\Omega_7$}\label{fig:mlc8}
\end{figure}

The proof of those results are quite similar to that given earlier for Fibonacci cubes and so is omitted.

\begin{proposition}\label{prop:rec-qF}
	For $n \ge 5$,
	\[
	q^F_k(\Omega_n) = q^F_{k-1}(\Omega_{n-2}) + q^F_k(\Omega_{n-3})
	\]
\end{proposition}

We list the first few of $Q^F(\Omega_n,x)$ as follows.
\begin{align*}
	Q^F(\Omega_0,x) &= 1 \\
	Q^F(\Omega_1,x) &= x \\
	Q^F(\Omega_2,x) &= 1+x \\
	Q^F(\Omega_3,x) &= 2x \\
	Q^F(\Omega_4,x) &= 1+x+x^2 \\
	Q^F(\Omega_5,x) &= 1+x+2x^2 %\\
%	Q^F(\Omega_6,x) &= 3x+x^2+x^3 \\
%	Q^F(\Omega_7,x) &= 1+2x+2x^2+2x^3 \\
%	Q^F(\Omega_8,x) &= 1+x+5x^2+x^3+x^5 \\
%	Q^F(\Omega_9,x) &= 4x+3x^2+3x^3+2x^4 \\
%	Q^F(\Omega_{10},x) &= 1+3x+3x^2+7x^3+x^4+x^5
\end{align*}

Obviously, Proposition~\ref{prop:rec-qF} yields the recurrence relation of $Q^F(\Omega_n,x)$.
\begin{proposition}\label{prop:rec-QF}
	For $n \ge 5$,
	\[
	Q^F(\Omega_n,x) = xQ^F(\Omega_{n-2},x) + Q^F(\Omega_{n-3},x).
	\]
\end{proposition}

Consider the generating function of $Q^F(\Omega_n,x)$ with $x=1,2$, we have two special sequences as follows.
\begin{corollary}\label{coro:rel-QF1-p}
	For $n \ge 0$, $Q^F(\Omega_n,1)$ is the $n+1$-th Padovan number too, i.e.
	\[
	Q^F(\Omega_n,1) = p_{n+1};
	\]
	%\end{corollary}
	%\begin{corollary}
%	For $n \ge 0$,
%	\[
%	Q^F(\Omega_n,2) = q_{n,0};
%	\]
	and for $n \ge 2$,
	\[
	Q^F(\Omega_n,2) = L_n.
	\]
\end{corollary}
%\begin{proof}
%	\[
%	1+y\sum_{n = 0}^\infty Q^F(\Omega_n,1) y^n = 1+\frac{1+y+y^2-y^3+y^3}{1-y^2-y^3}y = 1+\frac{1+y+y^2}{1-y^2-y^3}y = \frac{1+y}{1-y^2-y^3}
%	\]
%	%\end{proof}
%	%
%	%
%	%\begin{proof}
%	\[
%	\sum_{n = 0}^\infty Q^F(\Omega_n,2) y^n = \frac{1+y+y^2-y^3+2y^3}{1-2y^2-y^3} = \frac{(1+y)(1+y^2)}{(1 + y)(1-y-y^2)} = \frac{1+y^2}{1-y-y^2} = \frac{2-y}{1-y-y^2} + 1
%	\]
%\end{proof}

In addition, from Proposition~\ref{prop:rec-QF}, we obtain two generating functions.
\begin{theorem}\label{th:gf-QF}
	The generating functions of $Q^F(\Omega_n,x)$ and $p_n$ are
	\[
	\sum_{n = 0}^\infty Q^F(\Omega_n,x) y^n = \frac{(1 + y) (1 + y - y^3)}{1-xy^2-y^3} + (x-2)y,
	\]
	and
	\[
	\sum_{n=0}^\infty p_n y^n = \frac{1+y}{1-y^2(1+y)},
	\]
	respectively.
\end{theorem}
%\begin{proof}
%	By Proposition~\ref{prop:rec-QF}
%	\begin{align*}
%		\sum_{n=0}^\infty Q^F(\Omega_n,x) y^n &= \sum_{n=4}^\infty Q^F(\Omega_n,x) y^n + \sum_{n=0}^3 Q^F(\Omega_n,x) y^n \\
%		&= \sum_{n=4}^\infty (xQ^F(\Omega_{n-2},x)+Q^F(\Omega_{n-3},x)) y^n + \sum_{n=0}^3 Q^F(\Omega_n,x) y^n \\
%		&= xy^2\sum_{n=2}^\infty Q^F(\Omega_n,x) y^n + y^3\sum_{n=1}^\infty Q^F(\Omega_n,x) y^n + \sum_{n=0}^3 Q^F(\Omega_n,x) y^n \\
%		&= xy^2\sum_{n=0}^\infty Q^F(\Omega_n,x) y^n + y^3\sum_{n=0}^\infty Q^F(\Omega_n,x) y^n + \sum_{n=0}^3 Q^F(\Omega_n,x) y^n \\
%		&\phantom{{}={}} - xy^2\sum_{n=0}^1 Q^F(\Omega_n,x) y^n - y^3(Q^F(\Omega_0,x)) \\
%		&= (xy^2+y^3)\sum_{n=0}^\infty Q^F(\Omega_n,x) y^n + 1+y+y^2-y^3+xy^3
%	\end{align*}
%\end{proof}

We can obtain again the generating function of Padovan sequence as same as Corollary~\ref{coro:gf-pad}, so omit it.
%
%The Padovan numbers $\{p_n\}_{n=0}^\infty$ are defined as $p_0 = p_1 = p_2 =1$ and $p_n = p_{n-2} + p_{n-3}$ for $n \ge 3$, and the generating function of Padovan numbers is $\frac{1+y}{1-y^2-y^3}$. $1,1,1,2,2,3,4,5,7,9,\dots$
%
%\[
%(1+y) + (y^2+y^3(x-1))
%\]
%\[
%q^F_k(\Omega_n) = \sum_{j= \lfloor \frac{n+k+1}3 \rfloor}^{n+k} \binom jk \binom1{n+k-3j} + \sum_{j= \lfloor \frac{n-1}3 \rfloor}^{n} \sum_{m=0}^j \binom jm \binom1{n-2j-m-2} \binom{n-2j-m-2}{k-j+m} (-1)^{n-k-2j-2m-2}
%\]
%\[
%q^F_k(\Omega_n) = \sum_{j= 0}^{n+k} \left( \binom jk \binom1{n+k-3j} + \sum_{m=0}^j \binom jm \binom1{n-2j-m-2} \binom{n-2j-m-2}{k-j+m} (-1)^{n-k-2j-2m-2} \right)
%\]
%\[
%(1+y)(1+y^2)-y^3(x-2)
%\]
%\[
%q^F_k(\Omega_n) = \sum_{j = \lfloor \frac{n+k-1}3 \rfloor}^{\lfloor \frac{n+k}3 \rfloor} \binom jk \left( \binom1{n+k-3j-2} + \binom1{n+k-3j} \right) + \sum_{j = \lfloor \frac{n+k-2}3 \rfloor}^{\lfloor \frac{n+k-3}3 \rfloor} \binom{j}{n-2j-3} \binom1{n+k-3j-3}(-2)^{3j+4-n-k}
%\]
%\[
%q^F_k(\Omega_n) = \sum_{j = \lfloor \frac{n+k-1}3 \rfloor}^{\lfloor \frac{n+k}3 \rfloor} \left( \binom jk \left( \binom1{n+k-3j-2} + \binom1{n+k-3j} \right) - 2 \binom{j-1}{n-2j-1} \binom1{n+k-3j}(-2)^{3j-n-k} \right)
%\]
%
%\[
%(1-y^3) + (y+xy^3) + y^2
%\]
%
In general, a tedious calculation gives coefficients of $Q^F(\Omega_n,x)$ from Theorem~\ref{th:gf-QF}.
\begin{theorem}\label{th:QF}
	Let $\binom nk = 0$ if $n \notin \mathbb{Z}$. For $n \ge 2$,
	\[
	q^F_k(\Omega_n) = \binom{\frac{n+k+1}3-1}k + \binom{\frac{n+k}3-1}{k-1} + Y\left(\frac{n+k-1}3,k\right).
	\]
	In detail, for $n \ne 1$ and $k \ne 0,1$,
	\[
	q^F_k(\Omega_n) =
	\begin{cases}
	\binom{m-1}{k}, & \text{if } n+k=3m-1; \\
	\binom{m-1}{k-1}, & \text{if } n+k=3m; \\
	%\binom{m-1}{k-1} + \binom{m}{k} = 
	Y(m,k), & \text{if } n+k=3m+1.
	\end{cases}
	\]
%	For $n \ge 2$, $Q^F(\Omega_n,x) = $
%	\[
%	\sum_{j = \lfloor \frac{n+2}3 \rfloor}^{\lfloor \frac n2 \rfloor} \binom{j-1}{3j-n-1} x^{3j-n} + \sum_{j = \lfloor \frac{n+1}3 \rfloor}^{\lfloor \frac{n-1}2 \rfloor} \left(\binom{j-1}{3j-n} + \binom{j}{3j+1-n}\right) x^{3j+1-n} + \sum_{j = \lfloor \frac n3 \rfloor}^{\lfloor \frac{n-2}2 \rfloor} \binom{j}{3j+2-n} x^{3j+2-n}
%	\]

%	\[
%	q^F_k(\Omega_n) =
%	\begin{cases}
%	\binom{j-1}{3j-n-1}, & \text{if }3j=n+k; \\
%	\binom{j-1}{3j-n} + \binom{j}{3j+1-n}, & \text{if } 3j+1=n+k; \\
%	\binom{j}{3j+2-n}, & \text{if } 3j+2=n+k.
%	\end{cases}
%	\]
%	For $n \ne 1$ and $k \ne 0$,
%	\[
%	q^F_k(\Omega_n) =
%	\begin{cases}
%	\binom{m-1}{k-1}, & \text{if } n+k=3m,3m-1; \\
%	\binom{m-1}{k-1} + \binom{m}{k} = Y(m,k), & \text{if } n+k=3m+1.
%	\end{cases}
%	\]
%	\[
%	q^F_k(\Omega_n) =
%	\begin{cases}
%	\binom{\frac{n+k}3-1}{k-1}, & \text{if } n+k \equiv 0 \pmod 3; \\
%	\binom{\frac{n+k-1}3-1}{k-1} + \binom{\frac{n+k-1}3}{k}, & \text{if } n+k \equiv 1 \pmod 3; \\
%	\binom{\frac{n+k-2}3}{k}, & \text{if } n+k \equiv 2 \pmod 3.
%	\end{cases}
%	\]
\end{theorem}

\begin{table}[!h]
	\centering
	\caption{The coefficients of first few $Q^F(\Omega_n,x)$}\label{tab:coef-QF}
	\[
	\begin{matrix}%{*{7}{c}}
		1                                                                                \\
		0     & 1                                                                        \\
		1     & 1                                                                        \\
		0     & 2                                                                        \\
		1     & 1     & 1                                                                \\
		1     & 1     & 2                                                                \\
		0     & 3     & 1     & 1                                                        \\
		1     & 2     & 2     & 2                                                        \\
		1     & 1     & 5     & 1     & 1                                                \\
%		0     & 4     & 3     & 3     & 2                                                \\
%		1     & 3     & 3     & 7     & 1     & 1                                        \\
%		1     & 1     & 9     & 4     & 4     & 2                                        \\
%		0     & 5     & 6     & 6     & 9     & 1     & 1                                \\
%		1     & 4     & 4     & 16    & 5     & 5     & 2                                \\
%		1     & 1     & 14    & 10    & 10    & 11    & 1     & 1                        \\
%		0     & 6     & 10    & 10    & 25    & 6     & 6     & 2                        \\
%		1     & 5     & 5     & 30    & 15    & 15    & 13    & 1     & 1                \\
%		1     & 1     & 20    & 20    & 20    & 36    & 7     & 7     & 2                \\
%		0     & 7     & 15    & 15    & 55    & 21    & 21    & 15    & 1     & 1        \\
%		1     & 6     & 6     & 50    & 35    & 35    & 49    & 8     & 8     & 2        \\
%		1     & 1     & 27    & 35    & 35    & 91    & 28    & 28    & 17    & 1     & 1 \\
%		0     & 8     & 21    & 21    & 105   & 56    & 56    & 64    & 9     & 9     & 2 \\
	\end{matrix}%
	\]
\end{table}%

\begin{corollary}
%	If $3 \mid n$, $Q^F(\Omega_n,x)$ has $\lfloor \frac n2 \rfloor$ terms; $\lfloor \frac n2 \rfloor + 1$ otherwise.
	For $n \ge 4$, the number of terms with nonzero coefficients of $Q^F(\Omega_n,x)$ is $\lfloor \frac{n+5}3 \rfloor$.
\end{corollary}

\begin{corollary}%[up slanting]
	%	For $m \ge 0$ and $k \ge 0$,
	%	\[
	%	q^F_k(\Omega_{3m-k+2}) = \binom{m}k,
	%	\]
	%	\[
	%	q^F_k(\Omega_{3m-k+3}) = \binom{m}{k-1}
	%	\]
	%	For $m \ge 1$ and $k \ge 0$,
	%	\[
	%	q^F_k(\Omega_{3m-k+1})= q^F_k(\Omega_{3m-k}) + q^F_k(\Omega_{3m-k+2}) = 2q^F_k(\Omega_{3m-k}) + q^F_k(\Omega_{3m-k-1}) = Y(m,k).
	%	\]
	For $n \ne 1$ and $k \ne 0,1$,
	\[
	q^F_k(\Omega_{n-k}) = 
	\begin{cases}
	\binom{m-1}{k}, & \text{if } n=3m-1; \\
	\binom{m-1}{k-1}, & \text{if } n=3m; \\
	%\binom{m-1}{k-1} + \binom{m}k = 
	Y(m,k), & \text{if } n=3m+1.
	\end{cases}
	\]
%\end{corollary}
%\begin{corollary}[down slanting]
%	For $m \ge 1$,
%	\[
%	q^F_0(\Omega_{3m}) = 0.
%	\]
%	For $m \ge 0$ and $k \ge 0$,
%	\[
%	q^F_k(\Omega_{3m+2k+2}) = q^F_{k+1}(\Omega_{3m+2k+2}) = \binom km.
%	\]
%	For $mk \ne 0$,
%	\[
%	q^F_k(\Omega_{3m+2k}) = q^F_{k-1}(\Omega_{3m+2k}) = q^F_{k-1}(\Omega_{3m+2(k-1)+2}) = \binom{k-1}m,
%	\]
%	\[
%	q^F_k(\Omega_{3m+2k+1}) = \frac{m+2k}{m+k}\binom{m+k}{m} = \frac{m+2k}{m+k}\binom{m+k}{k};
%	\]
%	moreover, for $m \ge 1$,
%	\[
%	q^F_k(\Omega_{3m+2k+1}) = \frac{m+2k}{m}\binom{m+k-1}{m-1} = \frac{m+2k}{m}\binom{m+k-1}{k}.
%	\]
	And for $n \ne 1$ and $k \ne 0$,
	\[
	q^F_k(\Omega_{n+2k})=
	\begin{cases}
	\binom{m+k}k, & \text{if } n=3m-1; \\
	\binom{m+k-1}k, & \text{if } n=3m; \\
	Y(m+k,k), & \text{if } n=3m+1.
	\end{cases}
	\]
	Moreover, for $n \ge 0$, the $n$-th Padovan number can be given again.% by
%	\[
%	p_n = \sum_{j = \lfloor \frac{n-1}3 \rfloor}^{\lfloor \frac{n-1}2 \rfloor} \Bigg( \binom{j}{3j-n+3} + \binom{j+1}{3j-n+2} \Bigg).
%	\]
\end{corollary}
\begin{proof}
	Since $q^F_k(\Omega_{n-k})$ and $q^F_k(\Omega_{n+2k})$ are obvious, we need only calculate $p_n$.
	By Corollary~\ref{coro:rel-QF1-p} and Theorem~\ref{th:QF},
	\begin{align*}
	p_n &= \sum_j \binom{j-1}{3j-n} + \sum_j \binom{j-1}{3j-n+1} + \sum_j \binom{j}{3j-n+2} + \sum_j \binom{j}{3j-n+3} \\
	&= \sum_j \binom{j-1}{n-2j-1} + \sum_j \binom{j-1}{n-2j-2} + \sum_j \binom{j}{n-2j-2} + \sum_j \binom{j}{n-2j-3} \\
	&= \sum_j \binom{j}{n-2j-1} + \sum_j \binom{j+1}{n-2j-2} \\
	&= \sum_j \binom{j}{n-2j-1} + \sum_j \binom{j}{n-2j} \\
	&= \sum_j \binom{j+1}{n-2j}.
	\end{align*}
\end{proof}

In addition, from Yang Hui and Lucas triangles, we have a corollary.

\begin{corollary}
%	if $n=3m,3m-1$
%	\[
%	\sum_{k \ge 0}q^F_k(\Omega_{n-k}) = 2^{m-1}
%	\]
%	if $n=3m+1$
%	\[
%	\sum_{k \ge 0}q^F_k(\Omega_{n-k}) = 3 \times 2^{m-1}
%	\]
	For $n > 2$,
	\[
	\sum_{k \ge 0}q^F_k(\Omega_{n-k}) =
	\begin{cases}
	2^{m-1}, & \text{if } n=3m,3m-1; \\
	3 \times 2^{m-1}, & \text{if } n=3m+1.
	\end{cases}
	\]
%\end{corollary}
%
%\begin{corollary}
	And for $n \ge 6$,
	\[
	\sum_{k \ge 0} q^F_{n-4k,k} =
	\begin{cases}
	F_m, & \text{if } n=3m-1; \\
	F_{m-1}, & \text{if } n=3m; \\
	L_m, & \text{if } n=3m+1.
	\end{cases}
	\]
\end{corollary}

%\begin{corollary}[down]
%	Let $k \ge 0$ be a fixed integer, for $m \ge 0$,
%	\[
%	q^F_k(\Omega_{3m+2k+2}) = q^F_{k+1}(\Omega_{3m+2k+2}) = \binom mk;
%	\]
%	and for $mk \ne 0$,
%	\[
%	q^F_k(\Omega_{3m+2k+1}) = Y(m,k).
%	\]
%%	\[
%%	t_
%%	\]
%\end{corollary}

%\section*{Acknowledgements}
%\addcontentsline{toc}{section}{Acknowledgements}
%
%The authors are grateful to the referees for their careful reading and many valuable suggestions.

%\section*{References}

%\bibliographystyle{../../acmke}
%\bibliographystyle{abbrv}
%\bibliography{../../../JabRef/paper,../../../JabRef/book}

\end{document}